\documentclass[a4paper,12pt]{article}
\usepackage[english]{babel}
\usepackage{tikz}
\usetikzlibrary{matrix,arrows,decorations.markings}
\usepackage{amsmath,amsfonts,amssymb,amsthm,url,textcomp}
\usepackage{csquotes}
\usepackage{a4wide}
\usepackage[numbers]{natbib}
\usepackage{fancyhdr}
\usepackage{anyfontsize}
\usepackage{hyperref}
\usepackage{hhline}
\usepackage{leftidx}

\newtheorem{theorem}{Theorem}\numberwithin{theorem}{section}
\newtheorem{definition}[theorem]{Definition}
\newtheorem*{nottation}{Notation}
\newtheorem{lemma}[theorem]{Lemma}
\newtheorem{corollary}[theorem]{Corollary}
\newtheorem{proposition}[theorem]{Proposition}

\newtheorem{theoremm}{Theorem}\numberwithin{theoremm}{subsection}
\newtheorem{deffinition}[theoremm]{Definition}

\numberwithin{theoremmm}{subsubsection}

\theoremstyle{remark}
\newtheorem{remark}[theorem]{Remark}

\newcommand{\Rad}{\operatorname{Rad}}
\newcommand{\Aut}{\operatorname{Aut}}
\newcommand{\Alt}{\mathcal{A}}

\newcommand{\aff}{\mathrm{aff}}
\newcommand{\Aff}{\operatorname{Aff}}

\newcommand{\lcm}{\operatorname{lcm}}
\newcommand{\sh}{\operatorname{sh}}
\newcommand{\ord}{\operatorname{ord}}

\newcommand{\Hol}{\operatorname{Hol}}
\newcommand{\A}{\operatorname{A}}

\newcommand{\C}{\operatorname{C}}

\newcommand{\meo}{\operatorname{meo}}
\newcommand{\mao}{\operatorname{mao}}
\newcommand{\maffo}{\operatorname{maffo}}

\newcommand{\id}{\operatorname{id}}

\newcommand{\im}{\operatorname{im}}

\newcommand{\G}{\mathcal{G}}
\newcommand{\fin}{\mathrm{fin}}

\newcommand{\rel}{\mathrm{rel}}

\newcommand{\Mod}[1]{\ (\textup{mod}\ #1)}
\newcommand{\cha}{\operatorname{char}}

\begin{document}

\title{Finite groups with an automorphism of large order}

\author{Alexander Bors\thanks{University of Salzburg, Mathematics Department, Hellbrunner Stra{\ss}e 34, 5020 Salzburg, Austria. \newline E-mail: \href{mailto:alexander.bors@sbg.ac.at}{alexander.bors@sbg.ac.at} \newline The author is supported by the Austrian Science Fund (FWF):
Project F5504-N26, which is a part of the Special Research Program \enquote{Quasi-Monte Carlo Methods: Theory and Applications}. \newline 2010 \emph{Mathematics Subject Classification}: 20B25, 20D25, 20D45. \newline \emph{Key words and phrases:} Finite groups, Automorphisms, Abelian groups, Solvable groups, Solvable radical.}}

\date{\today}

\maketitle

\abstract{Let $G$ be a finite group, and assume that $G$ has an automorphism of order at least $\rho|G|$, with $\rho\in\left(0,1\right)$. Generalizing recent analogous results of the author on finite groups with a large automorphism cycle length, we prove that if $\rho>1/2$, then $G$ is abelian, and if $\rho>1/10$, then $G$ is solvable, whereas in general, the assumption implies $[G:\Rad(G)]\leq\rho^{-1.78}$, where $\Rad(G)$ denotes the solvable radical of $G$. Furthermore, we generalize an example of Horo\v{s}evski\u{\i} to show that in finite groups, the quotient of the maximum automorphism order by the maximum automorphism cycle length may be arbitrarily large.}

\section{Introduction}\label{sec1}

\subsection{Motivation and main results}\label{subsec1P1}

The purpose of this paper is to study finite groups that may be viewed as \enquote{extreme} with respect to their maximum automorphism order. More generally, many authors have studied finite groups satisfying \enquote{extreme} quantitative conditions of various kinds. We mention the following examples: A variety of papers deals with finite groups in which some automorphism raises some minimum fraction of elements to the $e$-th power for $e=-1,2,3$, see \cite{Mil29a,Mil29b,LM72a,Lie73a,LM73a,Mac75a,Pot88a,DM89a,Zim90a,Heg05a}. Wall classified the finite groups $G$ having more than $\frac{1}{2}|G|-1$ involutions \cite{Wal70a}, and this was extended to a classification of those $G$ with more than $\frac{1}{2}|G|-1$ subgroups of prime order by Burness and Scott \cite{BS09a}.

In \cite{Bor15a} and \cite{Bor15b}, we studied finite groups having an automorphism with a cycle of length at least $\rho|G|$, for some fixed $\rho\in\left(0,1\right)$. We proved that if $\rho>\frac{1}{2}$, then $G$ is abelian \cite[Theorem 1.1.7]{Bor15a}, and if $\rho>\frac{1}{10}$, then $G$ is solvable \cite[Corollary 1.1.2(1)]{Bor15b}. Furthermore, we showed that for any fixed value of $\rho$, the index of the solvable radical $\Rad(G)$ in $G$ is bounded from above in terms of $\rho$ if $G$ has such a long automorphism cycle \cite[Theorem 1.1.1(1)]{Bor15b}.

In this paper, we strengthen these results, replacing automorphism cycle lengths by automorphism orders:

\begin{theoremm}\label{mainTheo1}
Let $G$ be a finite group.

(1) If $G$ has an automorphism of order greater than $\frac{1}{2}|G|$, then $G$ is abelian.

(2) If $G$ has an automorphism of order greater than $\frac{1}{10}|G|$, then $G$ is solvable.

(3) For any $\rho\in\left(0,1\right)$, if $G$ has an automorphism of order at least $\rho|G|$, then $[G:\Rad(G)]\leq\rho^{E_1}$, where $E_1=(\log_{60}(6)-1)^{-1}=-1.7781\ldots$.
\end{theoremm}

A few comments relating this to the results on cycle lengths. Automorphisms $\alpha$ of finite groups having a cycle of length $\ord(\alpha)$ (following the terminology in \cite{GPS15a}, such cycles will be referred to as \textit{regular}) have been extensively studied by Horo\v{s}evski\u{\i} in \cite{Hor74a}. He gave examples of automorphisms of finite groups without a regular cycle (i.e., whose order is larger than the largest cycle length). In Section \ref{sec3}, we will generalize one of Horo\v{s}evski\u{\i}'s examples to show that in finite groups, the quotient of the maximum automorphism order by the maximum automorphism cycle length may be arbitrarily large.

On the other hand, Horo\v{s}evski\u{\i} also gave conditions on finite groups $G$ assuring that every automorphism of $G$ has a regular cycle, namely if $G$ is either semisimple (i.e., has no nontrivial solvable normal subgroups) \cite[Theorem 1]{Hor74a} or nilpotent \cite[Corollary 1]{Hor74a}. Note that in view of the latter, Theorem \ref{mainTheo1} implies the following: \textit{Every automorphism $\alpha$ of a finite group $G$ such that $l:=\ord(\alpha)>\frac{1}{2}|G|$ has a cycle of length $l$.} In particular, for $\rho>\frac{1}{2}$, the conditions \enquote{$G$ has an automorphism with a cycle of length $\rho|G|$.} and \enquote{$G$ has an automorphism of order $\rho|G|$.} are equivalent, and by \cite[Corollary 1.1.8]{Bor15a}, we obtain a complete classification of the pairs $(G,\alpha)$ where $G$ is a finite group and $\alpha$ an automorphism of $G$ such that $\ord(\alpha)>\frac{1}{2}|G|$.

We note that both in \cite{Bor15a} and in \cite{Bor15b}, we did not only study automorphisms, but a larger class of permutations on finite groups, so-called bijective affine maps:

\begin{deffinition}\label{affineDef}
Let $G$ be a finite group.

(1) For an element $x\in G$ and an endomorphism $\varphi$ of $G$, the \textbf{(left-)affine map of $G$ with respect to $x$ and $\varphi$} is the map $\A_{x,\varphi}:G\rightarrow G,g\mapsto x\varphi(g)$.

(2) The group of bijective affine maps of $G$ (which are just those $\A_{x,\varphi}$ where $\varphi$ is an automorphism of $G$) is denoted by $\Aff(G)$.
\end{deffinition}

We also had results for such maps, namely that a finite group $G$ having a bijective affine map with a cycle of length greater than $\frac{1}{4}|G|$ is solvable \cite[Theorem 1.1.1(2)]{Bor15b} and that $[G:\Rad(G)]$ is bounded from above in terms of $\rho$ in a finite group $G$ having a bijective affine map cycle of length at least $\rho|G|$. This can be strengthened in the same way as the results on automorphisms:

\begin{theoremm}\label{mainTheo2}
(1) Let $G$ be a finite group such that some $A\in\Aff(G)$ has order greater than $\frac{1}{4}|G|$. Then $G$ is solvable.

(2) Let $\rho\in\left(0,1\right)$ and let $G$ be a finite group such that some $A\in\Aff(G)$ has order at least $\rho|G|$. Then $[G:\Rad(G)]\leq\rho^{E_2}$, where $E_2=(\log_{60}(30)-1)^{-1}=-5.9068\ldots$.
\end{theoremm}

Finally, we remark that, just as for the results on cycle lengths, the constants $\frac{1}{2},\frac{1}{10}$ and $\frac{1}{4}$ in Theorems \ref{mainTheo1}(1,2) and \ref{mainTheo2}(1) respectively cannot be lowered further, as follows by considering maximum automorphism orders in finite dihedral groups (for Theorem \ref{mainTheo1}(1)) and in the alternating group $\Alt_5$.

\subsection{Notation}\label{subsec1P2}

By $\mathbb{N}^+$, we denote the set of positive integers. For a function $f$ and a set $M$, we denote by $f[M]$ the element-wise image of $M$ under $f$, and by $f_{\mid M}$ the restriction of $f$ to $M$. The order of a group element $g$ is denoted by $\ord(g)$. We write $N \cha G$ for \enquote{$N$ is a characteristic subgroup of $G$}. The finite field with $q$ elements is denoted by $\mathbb{F}_q$, and the logarithm with respect to a base $c>1$ by $\log_c$. We also use the following notation, most of which was already introduced in \cite{Bor15a} and \cite{Bor15b}:

\begin{nottation}\label{lambdaNot}
(1) Let $X$ be a finite set, $\psi$ a permutation on $X$. We denote by $\Lambda(\psi)$ the largest cycle length of $\psi$ and set $\lambda(\psi):=\frac{1}{|X|}\Lambda(\psi)$.

(2) For a finite group $G$, we define $\Lambda(G):=\max_{\alpha\in\Aut(G)}{\Lambda(\alpha)}$, $\lambda(G):=\frac{1}{|G|}\Lambda(G)$, $\Lambda_{\aff}(G):=\max_{A\in\Aff(G)}{\Lambda(A)}$ and $\lambda_{\aff}(G):=\frac{1}{|G|}\Lambda_{\aff}(G)$.

(3) For a finite group $G$, we denote by $\meo(G)$ the maximum element order of $G$ and set $\mao(G):=\meo(\Aut(G))$ and $\maffo(G):=\meo(\Aff(G))$.
\end{nottation}

\section{Finite dynamical systems and finite dynamical groups}\label{sec2}

This section gives a quick overview on some basic concepts and facts which we will need.

\begin{definition}\label{fdsDef}
(1) A \textbf{finite dynamical system} (\textbf{FDS}) is a finite set $S$ together with a function $f:S\rightarrow S$, a so-called \textbf{self-transformation of $S$}. It is called \textbf{periodic} if and only if $f$ is bijective.

(2) For FDSs $(S,f)$ and $(T,g)$, an \textbf{(FDS) homomorphism between $(S,f)$ and $(T,g)$} is a function $\eta:S\rightarrow T$ such that $\eta\circ f=g\circ\eta$. The \textbf{image of $\eta$}, denoted by $\im(\eta)$, is the FDS $(\eta[S],g_{\mid\eta[S]})$ An \textbf{(FDS) isomorphism} is a bijective FDS homomorphism.

(3) If $(S_1,f_1),\ldots,(S_r,f_r)$ are FDSs, their \textbf{(FDS) product} is defined as the FDS $(S_1\times\cdots\times S_r,f_1\times\cdots\times f_r)$, where $f_1\times\cdots\times f_r$ maps $(s_1,\ldots,s_r)\mapsto(f_1(s_1),\ldots,f_r(s_r))$.
\end{definition}

\cite[Sections 1--3]{Her05a} provides an introduction to the theory of FDSs. We will only need the following proposition summarizing some elementary facts on cycle lengths in FDS products:

\begin{proposition}\label{fdsProp}
Let $(S_1,f_1),\ldots,(S_r,f_r)$ be periodic FDSs.

(1) The cycle length of $(s_1,\ldots,s_r)\in S_1\times\cdots\times S_r$ under $f_1\times\cdots\times f_r$ is the least common multiple of the cycle lengths of the $s_i$ under $f_i$.

(2) If each $f_i$ has a regular cycle, then so does $f_1\times\cdots\times f_r$.\qed
\end{proposition}

Just as in \cite{Bor15a}, we will also work with the following notion:

\begin{definition}\label{fdgDef}
A \textbf{finite dynamical group} (\textbf{FDG}) is a finite group $G$ together with an endomorphism $\varphi$ of $G$.
\end{definition}

Hence any FDG is in particular an FDS, and an \textit{FDG homomorphism} is a map between the underlying groups of two FDGs which is both a group homomorphism and an FDS homomorphism. We will need the following elementary result in the proof of Theorem \ref{mainTheo1}(1):

\begin{corollary}\label{quotientCor}
Let $(G,\alpha)$ be a periodic FDG such that $\lambda(\alpha)>\frac{1}{2}$, and let $(Q,\beta)$ be a homomorphic image of it. If $\ord(\beta)=1$, then the group $Q$ is trivial.
\end{corollary}

\begin{proof}
Note that $\ord(\beta)=1$ implies that $g_1^{-1}g_2\in\ker{\eta}$ whenever $g_1,g_2\in G$ lie on the same cycle of $\alpha$. Since $\alpha$ has a cycle of length greater than $\frac{1}{2}|G|$ by assumption, it follows that $|\ker{\eta}|>\frac{1}{2}|G|$, whence $\ker{\eta}=G$ by Lagrange's theorem, and we are done.
\end{proof}

\section{On the quotient \texorpdfstring{$\mao(G)/\Lambda(G)$}{mao(G)/Lambda(G)}}\label{sec3}

The sole purpose of this section is to prove the following:

\begin{proposition}\label{quotientProp}
$\sup_G{\mao(G)/\Lambda(G)}=\infty$, where $G$ ranges over finite groups.
\end{proposition}

\begin{proof}
Fix $C\in\mathbb{N}^+$. Denote by $p_1,\ldots,p_{3C}$ the first $3C$ \textit{odd} primes in increasing order. For $i=1,\ldots,3c$, set $B_i:=\mathbb{Z}/p_i\mathbb{Z}$, and let $B:=\prod_{i=1}^{3C}{B_i}$. Observe that those automorphisms of $B$ that act by inversion on precisely one of the $B_i$ and identically on the others generate an elementary abelian $2$-subgroup $E\leq\Aut(B)$ of $\mathbb{F}_2$-dimension $3C$. Let $\psi:\mathbb{F}_2^{3C}\rightarrow E$ denote the $\mathbb{F}_2$-isomorphism mapping a vector $v$ to the automorphism $\psi(v)=:\alpha_v$ of $B$ acting identically on $B_i$ if the $i$-th component of $v$ is $0$, and otherwise by inversion.

Consider the $2$-dimensional subspace $U$ of $\mathbb{F}_2^3$ spanned by the vectors $(0,1,1)^t$ and $(1,0,1)^t$ together with the inclusion map $\iota:U\hookrightarrow\mathbb{F}_2^3$. Form an external direct sum $U_C$ of $C$ copies of $U$. The product of $C$ copies of $\iota$ (in the sense of Definition \ref{fdsDef}(3)) is an embedding $\iota':U_C\hookrightarrow\oplus_{i=1}^{C}{\mathbb{F}_2^3}$. Furthermore, there is an isomorphism $\sigma:\oplus_{i=1}^{C}{\mathbb{F}_2^3}\rightarrow\mathbb{F}_2^{3C}$ mapping the $i$-th standard basis vector, $i=1,2,3$, of the $j$-th summand, $j=1,\ldots,C$, of the source to the $(3(j-1)+i)$-th standard basis vector of $\mathbb{F}_2^{3C}$.

Let $W_C$ denote the image of $U_C$ under the embedding $\sigma\circ\iota'$ into $\mathbb{F}_2^{3C}$, and let $V_C\subseteq E$ denote the image of $W_C$ under $\psi$. For $i=1,\ldots,3C$, denote by $\pi_i:\mathbb{F}_2^{3C}\rightarrow\mathbb{F}_2$ the projection onto the $i$-th component. Observe that $W_C$ (resp. $V_C$) has the following two properties:

(i) For each $i=1,\ldots,3C$, there exists $v\in W_C$ such that $\pi_i(v)=1$. Hence for each $i=1,\ldots,3C$, there exists $\alpha_v\in V_C$ acting by inversion on $B_i$.

(ii) For each $v\in W_C$, $\pi_i(v)=0$ for at least $C$ values of $i\in\{1,\ldots,3C\}$. Thus each $\alpha_v\in V_C$ acts identically on at least $C$ of the $B_i$.

Let $G_C$ be the subgroup of $\Hol(B)=B\rtimes\Aut(B)$ generated by $B$ and $V_C$; then $G_C=B\rtimes V_C$. We will be done once we have showed that $\mao(G_C)/\Lambda(G_C)\geq 2^{C-1}$.

Let $\xi$ be an automorphism of $G_C$. Since the only elements of order $p_i$ in $G_C$ are the nontrivial elements of $B_i$, each $B_i$ (and hence $B$) is $\xi$-invariant; fixing a nontrivial element $b_i\in B_i$, we can write $\xi(b_i)=b_i^{k_i}$ with $k_i\in(\mathbb{Z}/p_i\mathbb{Z})^{\ast}$. Furthermore, since elements from different cosets of $B$ in $G_C$ act identically on different collections of the $B_i$, $\xi$ restricts to a permutation on each coset of $B$. Hence for studying the dynamics of $\xi$, we can partition $G_C$ into the cosets of $B$ and study the dynamics on each coset.

Let $\alpha_v\in V_C$. Observe that if $b$ is any element of $B$ having nontrivial $B_i$-component, where $i$ is such that $\pi_i(v)=0$, then $b\alpha_v$ does not have order $2$. Hence we can write $\xi(\alpha_v)=\prod_{i=1}^{3C}{b_i^{l_i}}\alpha_v$ with $l_i\in\mathbb{Z}/p_i\mathbb{Z}$ and $l_i=0$ if $\pi_i(v)=0$. It is not difficult to verify that the map $B\alpha_v\rightarrow B,b\alpha_v\mapsto b$, is an isomorphism between the FDSs $(B\alpha_v,\xi_{\mid B\alpha_v})$ and the FDS given by $B=\prod_{i=1}^{3C}{\mathbb{Z}/p_i\mathbb{Z}}$ together with the product of the affine self-transformations $\A_{l_i,k_i}$ of the $\mathbb{Z}/p_i\mathbb{Z}$ given by $x\mapsto k_ix+l_i$. Each $\A_{l_i,k_i}$ has a regular cycle (so that $\xi_{\mid B\alpha_v}$ has a regular cycle by Proposition \ref{fdsProp}(2)), and the order of $\A_{l_i,k_i}$ equals the order of $k_i\in(\mathbb{Z}/p_i\mathbb{Z})^{\ast}$ (which is a divisor of the even number $p_i-1$) if $k_i\not=1$, and otherwise, it equals the order of $l_i\in\mathbb{Z}/p_i\mathbb{Z}$ (which is a divisor of $p_i$). Hence we always have $\ord(\A_{l_i,k_i})\leq p_i$, and for those $i$ where $\pi_i(v)=0$, the order of $\A_{l_i,k_i}$ is a divisor of $p_i-1$. Since there are at least $C$ such $i$ by property (ii) above, this implies that the order (or largest cycle length) of $\xi_{\mid B\alpha_v}$ is bounded from above by $\frac{1}{2^{C-1}}\prod_{i,\pi_i(v)=0}{(p_i-1)}\prod_{i,\pi_i(v)=1}{p_i}\leq\frac{1}{2^{C-1}}\prod_{i=1}^{3C}{p_i}$.

On the other hand, considering the inner automorphism $\xi$ of $G_C$ with respect to the element $b_1\cdots b_{3C}$, $\xi$ fixes each $b_i$ (so that $k_i=1$ for all $i$ in the above notation), and $\xi(\alpha_v)=\prod_{i,\pi_i(v)=1}{b_i^2}\alpha_v$ for $v\in W_C$. In view of the above observations, this implies that every cycle length of $\xi$ is a product of some of the $p_i$, and by property (i) above, each $p_i$ occurs as a divisor of some cycle length. Hence $\ord(\xi)=\prod_{i=1}^{3C}{p_i}$. It follows that $\mao(G_C)/\Lambda(G_C)\geq\prod_{i=1}^{3C}{p_i}/(\frac{1}{2^{C-1}}\prod_{i=1}^{3C}{p_i})=2^{C-1}$.
\end{proof}

\begin{remark}
As mentioned earlier, the groups $G_C$ described in the proof of Proposition \ref{quotientProp} are generalizations of an example given by Horo\v{s}evski\u{\i}, see \cite[remarks after Corollary 1]{Hor74a}; Horo\v{s}evski\u{\i}'s example is our group $G_1$.
\end{remark}

\section{On the functions \texorpdfstring{$\mao_{\rel}$}{maorel} and \texorpdfstring{$\maffo_{\rel}$}{mafforel}}\label{sec4}

In this section, we study the functions assigning to each finite group the quotient of its maximum automorphism (resp. bijective affine map) order by the group order. We start with a very simple general lemma:

\begin{lemma}\label{abstractLem}
Let $f$ be a function from the class $\G^{\fin}$ of finite groups to the interval $\left(0,\infty\right)$ such that $f(G_1)=f(G_2)$ whenever $G_1\cong G_2$ and $f(G/\Rad(G))\geq f(G)$ for all finite groups $G$. Furthermore, assume that for finite semisimple groups $H$, $f(H)\to 0$ as $|H|\to\infty$; more explicitly, fix a function $g:\left(0,\infty\right)\rightarrow\left(0,\infty\right)$ such that for any $\rho\in\left(0,\infty\right)$, $f(H)<\rho$ whenever $H$ is a finite semisimple group with $|H|>g(\rho)$.

Then for any $\rho\in\left(0,\infty\right)$, if $G$ is a finite group such that $f(G)\geq\rho$, then $[G:\Rad(G)]\leq g(\rho)$.
\end{lemma}

\begin{proof}
By assumption, we have $f(G/\Rad(G))\geq f(G)\geq\rho$. Since $G/\Rad(G)$ is semisimple, this implies $[G:\Rad(G)]=|G/\Rad(G)|\leq g(\rho)$ by choice of $g$.
\end{proof}

Lemma \ref{abstractLem} summarizes our original approach to prove the weaker versions of Theorems \ref{mainTheo1}(3) and \ref{mainTheo2}(2) with cycle lengths instead of orders. Indeed, on the one hand, we observed that it follows from \cite[Lemma 2.1.4]{Bor15a} that $\lambda_{(\aff)}(G/N)\geq\lambda_{(\aff)}(G)$ for any finite group $G$ and $N\cha G$ (implying the first assumption $f(G/\Rad(G))\geq f(G)$ for these two $f$). On the other hand, assume that for some function $f:\G^{\fin}\rightarrow\left(0,\infty\right)$, we have $|H|\cdot f(H)\leq |H|^e$ for some $e\in\left(0,1\right)$ and all finite semisimple groups $H$. Then clearly, if $\rho\in\left(0,1\right)$ and $H$ is a finite semisimple group such that $|H|>\rho^{1/(e-1)}$, then $f(H)<\rho$, whence $g(\rho)$ from Lemma \ref{abstractLem} can be chosen as $\rho^{1/(e-1)}$. By \cite[Lemma 3.4]{Bor15b}, we have $\Lambda(H)\leq|H|^{\log_{60}(6)}$ and $\Lambda_{\aff}(H)\leq|H|^{\log_{60}(30)}$ for all finite semisimple groups $H$, thus explaining the exponents in \cite[Theorem 1.1.1]{Bor15b}.

Moreover, we know by \cite[Theorem 2.2.3]{Bor15b} that $\mao(H)=\Lambda(H)$ and $\maffo(H)=\Lambda_{\aff}(H)$ for all finite semisimple groups $H$. Hence by Lemma \ref{abstractLem} and the remarks from the last paragraph, Theorems \ref{mainTheo1}(3) and \ref{mainTheo2}(2) are clear once we have proved the following:

\begin{lemma}\label{radqLem}
Define functions $\mao_{\rel},\maffo_{\rel}:\G^{\fin}\rightarrow\left(0,\infty\right)$ by $\mao_{\rel}(G):=\frac{1}{|G|}\mao(G)$ and $\maffo_{\rel}(G):=\frac{1}{|G|}\maffo(G)$. Then $\mao_{\rel}(G/\Rad(G))\geq\mao_{\rel}(G)$ and $\maffo_{\rel}(G/\Rad(G))\geq\maffo_{\rel}(G)$ for all finite groups $G$.
\end{lemma}

Before proving Lemma \ref{radqLem}, we show:

\begin{lemma}\label{elAbLem}
Let $B$ be a finite elementary abelian group, and fix $\beta\in\Aut(B)$. Then $\lcm_{x\in B}{\ord(\A_{x,\beta})}\leq|B|$.
\end{lemma}

\begin{proof}
For $x\in B$, let $\sh_{\beta}(x):=x\beta(x)\cdots\beta^{\ord(\beta)-1}(x)\in B$. We observed in \cite{Bor15b} that $\ord(\A_{x,\beta})=\ord(\beta)\cdot\ord(\sh_{\beta}(x))$. Hence the least common multiple in question is either equal to $\ord(\beta)$, which is bounded from above by $|B|$ by \cite[Theorem 2]{Hor74a}, or to $p\cdot\ord(\beta)$, where $p$ is the prime base of $|B|$. Hence assume, for a contradiction, that some $\sh_{\beta}(x)$ is nontrivial and that $\ord(\beta)>\frac{1}{p}|B|$. Considering the primary rational canonical form of $\beta$ as an $\mathbb{F}_p$-automorphism (corresponding to a decomposition of $B$ into a maximal number of subspaces that are cyclic for $\beta$), we may assume by induction that $\beta$ can be represented by the companion matrix of $P(X)^k$ for some irreducible $P(X)\in\mathbb{F}_p[X]$. Note that all $\sh_{\beta}(x)$ are fixed points of $\beta$, and so $\beta$ has a nontrivial fixed point by assumption. This implies that for some nonzero $Q(X)\in\mathbb{F}_p[X]$ of degree less than $\deg(P(X)^k)$, we have $X\cdot Q(X)\equiv Q(X)\Mod{P(X)^k}$, or equivalently $P(X)^k\mid Q(X)\cdot (X-1)$. Since $P(X)^k\nmid Q(X)$, it follows that $P(X)\mid X-1$, and thus $P(X)=X-1$ by irreducibility. In view of the formula for the order of the companion matrix of $P(X)^k$ (first proved by Elspas \cite[Appendix II, 9]{Els59a}, see also \cite[Theorem 3.11]{LN97a} and \cite[Theorem 5 and remarks afterward]{Her05a}), it follows that $\ord(\beta)=p^{\lceil\log_p(k)\rceil}\leq p^{k-1}=\frac{1}{p}|B|$, a contradiction.
\end{proof}

\begin{proof}[Proof of Lemma \ref{radqLem}]
We only prove that $\maffo_{\rel}(G/\Rad(G))\geq\maffo_{\rel}(G)$, as the argument for $\mao_{\rel}$ is similar. The proof is by induction on $|\Rad(G)|$. For the induction step, fix $A=\A_{x,\alpha}\in\Aff(G)$ such that $\ord(A)=\maffo(G)$. Following the argument in \cite[proof of Theorem 2]{Hor74a}, we may fix a minimal $\alpha$-invariant elementary abelian normal subgroup $B$ of $G$. By the induction hypothesis, it is sufficient to show that $\maffo_{\rel}(G/B)\geq\maffo_{\rel}(G)$. Denoting by $\tilde{A}=\A_{\pi(x),\tilde{\alpha}}$ (where $\pi:G\rightarrow G/B$ is the canonical projection and $\tilde{\alpha}$ the induced automorphism on $G/B$) the induced affine map of $G/B$, we find that by \cite[Lemma 2.1.4]{Bor15a}, every cycle length of $A$ is a product of some cycle length of $\tilde{A}$ with some cycle length of a bijective affine map of $B$ of the form $\A_{b,\alpha_{\mid B}}$. Hence the order of $A$ divides the product of $\ord(\tilde{A})$ with $\lcm_{b\in B}(\A_{b,\alpha_{\mid B}})$. In particular, by Lemma \ref{elAbLem}, $\ord(A)\leq\ord(\tilde{A})\cdot |B|$. It follows that $\maffo_{\rel}(G)=\frac{1}{|G|}\ord(A)\leq\frac{1}{|G/B|}\ord(\tilde{A})\leq\maffo_{\rel}(G/B)$.
\end{proof}

\section{Proof of the main results}\label{sec5}

As explained in Section \ref{sec4}, Theorems \ref{mainTheo1}(3) and \ref{mainTheo2}(2) follow from Lemmata \ref{abstractLem} and \ref{radqLem} as well as the remarks between them, and deriving Theorem \ref{mainTheo1}(2) (resp. \ref{mainTheo2}(1)) from Theorem \ref{mainTheo1}(3) (resp. \ref{mainTheo2}(2)) is like in \cite[proof of Corollary 1.1.2, Section 3]{Bor15b}. Hence it only remains to prove Theorem \ref{mainTheo1}(1).

Fix an automorphism $\alpha$ of $G$ such that $\ord(\alpha)>\frac{1}{2}|G|$. We prove that $G$ is abelian by induction on $|G|$. For the induction step, observe that $G$ cannot be semisimple, since otherwise, by \cite[Theorem 1]{Hor74a}, $\alpha$ would have a regular cycle and hence $G$ would be abelian by \cite[Theorem 1.1.7]{Bor15a}, contradicting its semisimplicity.

Like in the proof of Lemma \ref{radqLem}, following the argument in \cite[proof of Theorem 2]{Hor74a}, we fix a minimal $\alpha$-invariant elementary abelian normal subgroup $B$ of $G$. We may of course assume that $B$ is proper in $G$. Denote by $\tilde{\alpha}$ the induced automorphism of $G/B$, set $m:=\ord(\tilde{\alpha})$, $n:=\ord(\alpha_{\mid B})$ and denote by $C$ the set of fixed points of $\alpha^m$ in $B$. Horo\v{s}evski\u{\i} proceeded to show that either $C=\{1\}$ or $C=B$ (by minimality of $B$) and to derive upper bounds for $\ord(\alpha)$ in both cases, which imply that $\ord(\alpha)\leq m\cdot|G/B|$ in any case and thus $m\geq\ord(\alpha)/|G/B|=|B|\cdot\ord(\alpha)/|G|>\frac{1}{2}|B|$, whence $G/B$ is abelian by the induction hypothesis.

In particular, we have $G'\leq B$ and $\lambda(\tilde{\alpha})>\frac{1}{2}$ by \cite[Corollary 1]{Hor74a}. Consider the homomorphism $\varphi:G\rightarrow\Aut(B)$ corresponding to the conjugation action of $G$ on $B$. Since $B$ is abelian, we have $B\leq\ker(\varphi)$, and so there is a homomorphism $\overline{\varphi}:G/B\rightarrow\Aut(B)$ such that $\overline{\varphi}\circ\pi_B=\varphi$, where $\pi_B:G\rightarrow G/B$ is the canonical projection.

Now the kernel of $\overline{\varphi}$ consists by definition of those $\pi_B(g)\in G/B$ such that $gB\subseteq\C_G(B)$. Clearly, since $B$ is $\alpha$-invariant, so ist $\C_G(B)$, and thus $\ker(\overline{\varphi})$ is $\tilde{\alpha}$-invariant. It follows that there exists an automorphism $\overline{\alpha}$ on the image $\overline{\varphi}(G/B)\leq\Aut(B)$ such that the following diagram commutes:

\begin{center}
\begin{tikzpicture}
\matrix (m) [matrix of math nodes, row sep=3em,
column sep=3em]
{ G/B & G/B \\
\overline{\varphi}(G/B) & \overline{\varphi}(G/B) \\
};
\path[->]
(m-1-1) edge node[above] {$\tilde{\alpha}$} (m-1-2)
(m-1-1) edge node[left] {$\overline{\varphi}$} (m-2-1)
(m-1-2) edge node[right] {$\overline{\varphi}$} (m-2-2)
(m-2-1) edge node[below] {$\overline{\alpha}$} (m-2-2);
\end{tikzpicture}
\end{center}

In other words, the FDG $(\overline{\varphi}(G/B),\overline{\alpha})$ is the image of the FDG $(G/B,\tilde{\alpha})$ under the FDG homomorphism $\overline{\varphi}:(G/B,\tilde{\alpha})\rightarrow(\overline{\varphi}(G/B),\overline{\alpha})$. By this definition of $\overline{\alpha}$, it is clear that $\ord(\overline{\alpha})\mid\ord(\tilde{\alpha})=m$.

We give an alternative definition of $\overline{\alpha}$. The element $\overline{\varphi}(gB)\in\overline{\varphi}(G/B)$, which is by definition the restriction of conjugation by $g$ to $B$, is mapped by $\overline{\alpha}$ to $\overline{\alpha}(\overline{\varphi}(gB))=\overline{\varphi}(\tilde{\alpha}(gB))=\overline{\varphi}(\alpha(g)B)$, which is the restriction of conjugation by $\alpha(g)$ to $B$. But this implies that $\overline{\alpha}$ is the restriction of conjugation by $\alpha_{\mid B}$ in $\Aut(B)$ to its subgroup $\overline{\varphi}(G/B)$. In particular, $\ord(\overline{\alpha})\mid\ord(\alpha_{\mid B})=n$.

We now distinguish two cases. First, assume that $B$ is cyclic. Then $\Aut(B)$ is abelian, and so by the second definition of $\overline{\alpha}$, it is clear that $\overline{\alpha}=\id_{\overline{\varphi}(G/B)}$. By Corollary \ref{quotientCor}, this implies that $\overline{\varphi}$ is the trivial homorphism $G/B\rightarrow\Aut(B)$, and by definition of $\overline{\varphi}$, this just means that $B\leq\zeta G$. In particular, we have $G'\leq\zeta G$, whence $G$ is nilpotent of class $2$. By \cite[Corollary 1]{Hor74a}, this implies that $\lambda(\alpha)=\ord(\alpha)>\frac{1}{2}|G|$, and so $G$ is abelian by \cite[Theorem 1.1.7]{Bor15a}.

Now assume that $B\cong(\mathbb{Z}/p\mathbb{Z})^n$ for some prime $p$ and $n\geq 2$. By the argument in \cite[proof of Theorem 2]{Hor74a}, if $C=B$, we have $\ord(\alpha)\leq m\cdot p\leq |G/B|\cdot\frac{1}{p}|B|=\frac{1}{p}|G|$, a contradiction. Hence $C=\{1\}$, whence by \cite[Lemma 3a]{Hor74a}, we have $\ord(\alpha)=\lcm(m,n)$. If $\gcd(m,n)>1$, it follows that $\ord(\alpha)\leq\frac{1}{2}\cdot m\cdot n\leq\frac{1}{2}\cdot|G/B|\cdot|B|\leq\frac{1}{2}|G|$, a contradiction. Therefore, $\gcd(m,n)=1$, which implies that $\ord(\overline{\alpha})=1$. Now repeat the argument from the first case to conclude the proof.\qed

\end{document}